\documentclass{amsart}
\usepackage{amssymb, amsthm, amsmath, enumerate, stmaryrd, mathrsfs}
\usepackage[lmargin=1in,rmargin=1in]{geometry}

\usepackage{wrapfig,tikz, tikz-cd}
\usetikzlibrary{arrows}

\usepackage[parfill]{parskip}

\newtheorem{theorem}{Theorem}
\newtheorem*{theorem*}{Theorem}
\newtheorem{lemma}{Lemma}
\newtheorem{proposition}{Proposition}
\newtheorem{corollary}{Corollary}

\theoremstyle{definition}
\newtheorem{remark}{Remark}
\newtheorem*{example*}{Example}

\newcommand{\spec}{{\rm spec}}
\newcommand{\sub}{\subseteq}
\newcommand{\xto}{\xrightarrow}

\newcommand{\CC}{\mathbb{C}}
\newcommand{\ZZ}{\mathbb{Z}}
\newcommand{\Der}{{\rm Der}_{\mathbb{C}}}

\begin{document}

\title{Matrix Factorisations Arising From Well-Generated Complex Reflection Groups}



\author{Benjamin Briggs}
\thanks{At the time this work was completed the author was affiliated with the University of Toronto.}

\maketitle

\begin{abstract}
We discuss an interesting duality known to occur for certain complex reflection groups, namely the \emph{duality groups}. Our main construction yields a concrete, representation theoretic realisation of this duality. This allows us to naturally identify invariant vector fields with vector fields on the orbit space, for the action of a duality group. As another application, we construct matrix factorisations of the highest degree basic invariant which give free resolutions of the module of K\"{a}hler differentials of the coinvariant algebra $A$ associated to such a reflection group. From this one can explicitly calculate the dimension of each graded piece of $\Omega_{A/\CC}$ and of $\Der(A,A)$, adding a new formula to the numerology of reflection groups. This applies for instance when $A$ is the cohomology of any complete flag manifold, and hence has geometric consequences.
\end{abstract}

\section{Introduction}

To a complex reflection group $W$ acting on a vector space $V$ one may associate the algebra of coinvariants $A=S/(S^W_+)$, where $S=\CC[V]$. By the classical Chevalley-Shephard-Todd theorem $A$ is a zero dimensional complete intersection, and as such $A$ enjoys good homological properties. Our aim is to investigate how the context at hand might further influence the homological behaviour  of $A$. For instance, when $W$ is well-generated we will be able to precisely calculate the dimensions of the graded components of $\Der(A,A)$.

Let us specialise to the case that $W$ is the Weyl group of a semi-simple algebraic group $G$. If we choose a Borel subgroup $B$ and a maximal torus $T$ in $B$, then $W$ acts on $V={\rm Lie}(T)$ as a reflection group. According to the famous ``Borel picture" the cohomology of the complete flag manifold $X=G/B$ is canonically isomorphic to the corresponding algebra of coinvariants $A\cong {\rm H}^{*}(X;\CC)$. The graded algebra $A$ actually completely determines the rational homotopy type of $X$ (because $X$ is a formal space). For this reason, homological facts about $A$ can be translated into homotopy theoretic information about $X$.
t
Specialising further still, one question which motivated the constructions below is as follows. The symmetric group $W=\mathfrak{S}_{n+1}$ acts on $S=\CC[x_0,...,x_n]$ by permuting the variables. The invariant subring $S^W$ is the polynomial algebra $\CC[\sigma_1,...,\sigma_{n+1}]$ on the elementary symmetric polynomials, so the algebra of coinvariants is $A=S/(\sigma_1,...,\sigma_{n+1})$. The module of K\"{a}hler differentials $\Omega^1_{A/\CC}$ may be presented as the cokernel of the Vandermonde matrix
\[
\begin{pmatrix}
1& x_0 &\cdots & x_0^{n}\\
\vdots & & & \vdots\\
1& x_n &  \cdots& x_n^{n} \\
\end{pmatrix}
: A^{n+1}\to A^{n+1}\to \Omega^1_{A/\CC}\to 0.
\]
How can one continue this to a resolution of $\Omega^1_{A/\CC}$? Experiments performed by Ragnar-Olaf Buchweitz in Macaulay 2 suggested that the resolution might be $2$-periodic.
Remarkably, this is always the case: the Vandermonde matrix is part of a matrix factorisation of $\sigma_{n+1}$ over the algebra $S/(\sigma_1,...,\sigma_{n})$. This curious combinatorial fact does not seem to appear in the literature. Knowing this, one can deduce facts about the space  ${Fl}_n(\CC)$ of complete flags $0=V^0\sub V^1\sub\cdots\sub V^n=\CC^n$ in $\CC^n$.

At the core of these considerations is a duality  discovered by Orlik and Solomon \cite{MR575083} which takes place in the invariant theory of certain complex reflection groups, all Weyl groups included. These are the duality groups, otherwise known as the well-generated reflection groups. We show in theorem \ref{thm1} that this duality is realised by a concrete, representation theoretic pairing.

As a first application we show that duality groups are characterised by the existence of an isomorphism $\Der(S,S)^W\cong \Der(S^W,S^W)$ (which decreases degree by the Coxeter number), see theorem \ref{derthm}.

In section \ref{resolutions} we use the pairing to construct $2$-periodic resolutions of the $A$ modules $\Der(A,A)$ of derivations and $\Omega^1_{A/\CC}$ of K\"{a}hler differentials, which come from a matrix factorisation of the highest degree basic invariant. This was our motivation for the construction of the duality pairing. We use this $2$-periodic resolution to write down explicit formulas for the dimensions of all the graded components of $\Der(A,A)$ and $\Omega^1_{A/\CC}$. These surprisingly attractive formulas appear in theorem \ref{HilbertSeries}.

Finally, in section \ref{flagvarieties} we return to our geometric motivation and deduce some facts about flag manifolds. 

Thank you to the anonymous referee for an extremely detailed and helpful report, which lead to a lot of improvements in the paper (in particular to the two remarks \ref{BessisRemark} and \ref{ShephardRemark} are due to the referee's comments). Thanks also to Vincent G\'elinas for comments on an early draft of this paper.

\section{Preliminaries}\label{prelim}

We begin by reviewing some well-known facts from the invariant theory of complex reflection groups, all of which can be found in \cite{MR2542964} or \cite{MR1249931}. The reader will likely find these references considerably less dense than this section, which largely serves to fix notations and conventions. Almost everything below works for an arbitrary field of characteristic zero, but we will work over $\CC$. 

Let $W$ be a complex reflection group acting irreducibly on a complex vector space $V$ of dimension $n$. Then $W$ acts on the ring $S=\CC[V]={\rm Sym}_{\CC}(V^*)$ of polynomials on $V$. We grade  $S$ in the usual way by placing $V^*$ in degree $1$. 

Let $S^W$ be the algebra of invariants in $S$. The ideal  $I=S\cdot S^W_+$ of $S$ is called the Hilbert ideal, and the quotient $A=S/I=S\otimes_{S^W} \CC$ is called the algebra of coinvariants.

One special feature of the theory of complex reflection groups is that $I$ has a canonical $\CC W$-module complement in $S$, defined as follows. The graded dual $S^*\cong {\rm Sym}_{\CC}(V)$ of $S$ may be thought of as the algebra of differential operators on $S$, with composition as multiplication. 
Of course $W$ acts on $S^*$, and one can define $H$ to be the space of polynomials in $S$ which are killed by every non-constant invariant differential operator in $(S^*)^W_+$. The upshot is that the composition $H\hookrightarrow S\twoheadrightarrow A$ is an isomorphism, see e.g.  \cite[Corollary 9.37]{MR2542964}. 
 Thus $H$ provides  canonical  representatives for elements of $A$ in $S$; it is known as the space of harmonic polynomials. 
 Everything below would work fine just choosing a fixed $W$-invariant complement to $I$ in $S$, we use $H$ only for aesthetic reasons.


The ring of invariants is itself a polynomial ring: the canonical map ${\rm Hom}_W(V,H)\cong (V^*\otimes H)^W\to S$ given by multiplying elements of $V^*$ and $H$ induces an isomorphism
\begin{equation}\label{CST-theorem}
{\rm Sym}_{\CC}({\rm Hom}_W(V,H))\xto{\,\sim\, } S^W.
\end{equation}
This is one half of the classical Chevalley-Shephard-Todd Theorem, which says further that complex reflection groups are characterised by this invariant theoretic property. 
The theorem means we may interpret the quotient $V{\!/\!\!/} W$ as the vector space  ${\rm Hom}_W(V,H)^*$. This is not the usual way of stating the Chevalley-Shephard-Todd Theorem, but it can easily be deduced from it.  An equivalent statement of the theorem is that multiplication $H\otimes S^W\to S$ is an isomorphism of $\CC W$-modules. 


By a theorem of Steinberg \cite{MR0167535}, the critical locus $C$ of the quotient map $V\to V{\!/\!\!/} W$ is precisely the union of the reflecting hyperplanes. The regular locus is by definition $V^{\rm reg}=V\setminus C$, and elements of $V^{\rm reg}$ are called regular vectors. By Steinberg's theorem there is a natural surjection from the so-called braid group $B(W)=\pi_1(V^{\rm reg}{\!/\!\!/} W)$ onto $W$, coming from the short exact sequence of the covering $V^{\rm reg}\to V^{\rm reg}{\!/\!\!/} W$. The image of the critical locus in $V{\!/\!\!/} W$ is called the discriminant locus, and this is cut out by a single reduced polynomial $\Delta$ in $S^W$.

The degrees in which ${\rm Hom}_W(V,H)$ is nonzero (counted with multiplicity) are called the degrees of $W$, denoted by $d_1\leq \cdots \leq d_n$. They are one more than the degrees of the $V$-isotypic components in $H$. The codegrees $d_1^*\geq \cdots \geq d_n^*$ are by definition the degrees in which ${\rm Hom}_W(V^*,H)$ is nonzero, so these are one less than the degrees of the $V^*$-isotypic components in $H$. The highest degree $d_n$ plays an important role, it will be denoted throughout simply by $d$. This is sometimes called the Coxeter number of $W$.

More generally, the degrees of a representation $M$, denoted $d_1^M\leq \cdots \leq d_i^M$, are the degrees in which ${\rm Hom}_W(M,H)$ is nonzero. The codegrees of $M$ are the degrees of $M^*$.

A fixed choice of homogeneous generators $S^W=\CC[f_1,...,f_n] $ for the invariant algebra is called a set of basic invariants. They can be taken as the image of a homogeneous basis of ${\rm Hom}_W(V,H)$ under the canonical map to $S^W$, and therefore we can assume the degree of $f_i$ is $d_i$ (any other choice of basic invariants will have the same degree, up to reordering).

The Chevalley-Shephard-Todd Theorem is usually stated in terms of the existence of a set  of basic invariants, but note that the statement (\ref{CST-theorem}) provides a \emph{canonical} vector space of basic invariants, which come from ``integrating'' harmonic polynomials. For this reason, nothing we do below depends on a  choice of basic invariants.

In any case, it follows that the Hilbert ideal $I$ is generated by a regular sequence of length $n$, namely $I=(f_1,...,f_n)$. In particular, $A=S/I$ is a zero dimensional complete intersection.

We say that $W$ is is a duality group if it happens that $d_i+d_i^*=d$ for all $i$. This is a priori a property of $A$ as a graded representation of $W$ alone, but we will see that the commutative algebra of $A$ is heavily influenced by this condition.

We say that $W$ is well-generated if it can be generated by exactly $n$ complex reflections. Classical convex geometry reveals this to be the case for all real reflection groups. That is, those for which the action of $W$ on $V$ is the complexification of an action of $W$ on a real vector space. A real reflection group is precisely a finite Coxeter group equipped with the complexification of its Tits representation.

The following theorem was first observed by Orlik and Solomon in \cite{MR575083} using a case-by-case check of the Shephard-Todd classification.

\begin{theorem*}
A complex reflection group is a duality group exactly when it is well-generated.
\end{theorem*}

In \cite{MR1856398} Bessis proves that duality groups are well-generated without using the classification. Consider a generic line $L$ of direction $f_n$ in $V{\!/\!\!/} W$ (i.e. parallel to $\{f_1=\cdots =f_{n-1}=0\}$). In a duality group the direction of $L$ is uniquely determined. (If we use our convention that the basic invariants come from ${\rm Hom}_W(V,H)$ as in (\ref{CST-theorem}), then in fact $f_n$ is uniquely determined up to a scalar as well.) Bessis observes that lemma \ref{lem1} below, along with regularity of $d$ for duality groups (see \ref{regularcriterion} below), implies that $L$ intersects the discriminant hypersurface transversally exactly $({\rm deg} \Delta)/d=(d_1+\cdots + d_n +d_1^*+\cdots + d_n^*)/d=n$ times. He then shows that after removing these points, the inclusion of the punctured line $L^{\text{reg}}$ into $V^{\text{reg}}{\!/\!\!/} W$ is surjective on fundamental groups, producing $n$ generators for the braid group, and hence for $W$. Despite this pleasing argument, there still seems to be no conceptual proof of the converse. The considerations below may shed light on this.

If $k$ is an integer let us denote by $I_k^W$ the ideal in $S^W$ generated by homogeneous polynomials with degree not divisible by $k$. The corresponding ideal in $S$ is denoted $I_k=S\cdot I^W_k$. In particular $I_0=I$ is the Hilbert ideal. The notation is consistent in that $I_k^W$ consists exactly of the invariant polynomials in $I_k$. Note that $I_k$ is generated by a regular sequence, namely, by those basic invariants it contains. $I^W_k$ is always radical but $I_k$ need not be.

We will make use of Springer's theory of regular elements \cite{MR0354894}. Given an element $g$ of $W$ and a complex number $\zeta$ let us write $V(g,\zeta)$ for the $\zeta$-eigenspace of $V\xto{\, g\, } V$. The element $g$ is called regular when it has an eigenspace  $V(g,\zeta)$ containing a regular vector, and an integer $k$ is a regular number when it is the order of a regular element, or equivalently if some  $V(g,\zeta)$ contains a regular vector with $\zeta$ being a primitive $k$th root of unity. The following is proposition 3.2 of \cite{MR0354894}.

\begin{lemma}\label{lem1} Let $\zeta$ be a primitive $k$th root of unity. Then
\[
\bigcup_gV(g,\zeta)=Z(I_k).
\]
\end{lemma}

Here $Z(I_k)$ denotes the zero set of $I_k$, or equivalently $\bigcap_{k\nmid d_i} \{ f_i=0\}$. 
As Bessis observes in \cite{MR1856398}, it follows that $k$ is regular if and only if $\Delta$ is not in $I_k^W$.

The other fact about regular vectors we need is the following lemma, which generalises the fact that $A$ carries the regular representation. In fact, it provides a family of  left $W$-module isomorphisms between $A$ and $\CC W$ paramatrised by $V^\text{reg}$.

\begin{lemma}\label{lem2} If $M$ is any representation of $W$ and $v$ is a regular vector then composing with the evaluation $H\hookrightarrow S\xto{{\rm ev}_v} \CC$ at $v$ results in an isomorphism ${\rm Hom}_W(M,H)\xto{\,\sim\,} M^*$ of vector spaces. 
\end{lemma}

This is lemma 2.6 from \cite{MR575083}. 
Note that this is not surprising when $M=V$, because the quotient $V\to V{\!/\!\!/} W$ is a local isomorphism at a regular vector, and the lemma simply provides an isomorphism between the two cotangent spaces.


\section{The Duality Group Pairing} \label{pairing}

For any representation $M$ of $W$ there is a pairing which lands in $S^W$:
\[
\begin{tikzcd}
                {\rm Hom}_W(M,H)\otimes {\rm Hom}_W(M^*,H) \ar[d] \\
                {\rm Hom}_W(M,S)\otimes {\rm Hom}_W(M^*,S) \ar[d] \\
                {\rm Hom}_W(M\otimes M^*,S) \ar[d]\\
                {\rm Hom}_W(\CC,S) = S^W
\end{tikzcd}
\]
the second arrow using multiplication in $S$ and the third coming from the $W$-invariant diagonal $\CC\to M\otimes M^*$. Note that all of these maps preserve the gradings. From $S^W$ we can evaluate at a vector $v$ to land in $\CC$. The following diagram then commutes
\[
     \begin{tikzcd}
                {\rm Hom}_W(M,H)\otimes {\rm Hom}_W(M^*,H)       \ar[d] \ar[r]     & S^W  \ar[d, "{\rm ev}_v"]   \\
                M^*\otimes M   \ar[r] & \CC
     \end{tikzcd}
\]
where the leftmost map comes from applying lemma \ref{lem2} to both $M$ and $M^*$. If $v$ is regular lemma \ref{lem2} says this is an isomorphism. In this case going around the lower-left  clearly results in a non-degenerate pairing, so we obtain:

\begin{proposition}\label{perfect} When $v$ is a regular vector the pairing
\[
{\rm Hom}_W(M,H)\otimes {\rm Hom}_W(M^*,H)\to S^W \xto{\ \text{ev}_v \ } \CC
\]
is perfect.
\end{proposition}

Of course this depends only on the image of $v$ in $V{\!/\!\!/} W$. Next we think about how to preserve some of the grading information so that we can relate the degrees and codegrees of $M$.

Lemma \ref{lem2} can be restated algebraically as saying that if $k$ is a regular number we can find a regular vector $v$ for which evaluation at $v$ factors through the quotient:
\[
S^W \xto{\ \ \ } S^W/I_k \xto{\ \text{ev}_v \ } \CC.
\]
If one considers $S^W$ and $\CC$ to be graded by $\ZZ/k\ZZ$, with $\CC$ in degree $0$, then another way of stating this is that evaluation $S^W\to \CC$ at $v$ preserves this grading. Combining these remarks with proposition 1 gives the following:

Suppose $k$ is a regular number for $W$. For any representation $M$ of $W$ there is a permutation $\pi$ of $\{1,...,{\rm dim}\,M \}$ such that for each $i$
\[
d_i^M+d_{\pi i}^{M^*} =0 \quad {\rm modulo}\  k.
\]
\subsection{}\label{regularcriterion} This was observed by Orlik and Solomon by comparing the eigenvalues of the $W$ action on $M$ and on $M^*$ \cite[(5.1)]{MR575083}. There is a converse to this result: if such a permutation exists for $M=V$ then $k$ is a regular number. Even stronger, it is proven in \cite{MR1756877} that $k$ is regular as long as it divides exactly as many degrees as codegrees. The same fact is given a case free proof in \cite{MR1990017}.

Now we turn to the case of duality groups. This is the best situation for controlling the grading. We are going to start using the notation $M(n)$ for the $n$-fold shift of a graded vector space $M$, with $M(n)_i=M_{i-n}$.

After pairing ${\rm Hom}_W(V,H)\otimes {\rm Hom}_W(V^*,H)\to S^W$ we can canonically project onto a highest degree basic invariant
\[
S^W = {\rm Sym}_{\CC}{\rm Hom}_W(V,H) \twoheadrightarrow  {\rm Hom}_W(V,H)_d.
\]
Since $A$ and $H$ are canonically isomorphic, we switch freely between them below. 

\begin{theorem}\label{thm1} When $W$ is a duality group ${\rm Hom}_W(V,A)_d$ is one dimensional and the pairing
\[
{\rm Hom}_W(V,A)\otimes {\rm Hom}_W(V^*,A) \to {\rm Hom}_W(V,A)_d
\]
is perfect.
\end{theorem}
Since ${\rm Hom}_W(V,A)_d$ is concentrated in degree $d$ this pairing is a concrete, representation-theoretic witness to the duality $d_i+d^*_i=d$. 

By transposing the pairing, there is a natural map $D: {\rm Hom}_W(V^*,A)\to {\rm Hom}_W(V,A)^*(d)$ which is an isomorphism if and only if $W$ is a duality group. This isomorphism will be very useful to us below.

\begin{proof}

Firstly, ${\rm Hom}_W(V,H)_d$ is one dimensional because ${\rm Hom}_W(V^*,H)_{0}$ clearly is. This implies that $d\nmid d_i$ precisely for $i\neq n$. In other words, the zero-set $Z(I_d)$ of lemma \ref{lem1} is $\{f_1=...=f_{n-1}=0\}$. The regularity criterion mentioned in paragraph \ref{regularcriterion} implies that the highest degree is always regular in a duality group. So we may invoke lemma \ref{lem1} to find a regular vector $v$ in $Z(I_d)$.

The point will be that the triangle below almost commutes, enough to make the pairing along the top perfect.
\[
\begin{tikzcd}[row sep=3mm]
     & & {\rm Hom}_W(V,H)_d \ar[dd, "\wr"] \\
			{\rm Hom}_W(V,H)\otimes {\rm Hom}_W(V^*,H)  \ar[r]   & S^W  \ar[ur] \ar[dr, "{\rm ev}_v"]  & \\
			& & \CC
\end{tikzcd}
\]
The vertical map is any choice of isomorphism, that is, a choice of highest degree basic invariant $f_n$ (there is a canonical one, with $f_n(v)=1$).
 
The lower pairing is perfect by proposition \ref{perfect}. Suppose $\phi\in {\rm Hom}_W(V,H)$ and $\psi\in  {\rm Hom}_W(V^*,H)$ pair nontrivially in this way. Say $\phi\otimes \psi$ goes to $f\in S^W$ and $f(v)\neq 0$. Since $f_1(v)=...=f_{n-1}(v)=0$ this $f$ involves a monomial $f_n^k$. For degrees reasons, using the fact that $d_i+d^*_j<d_i+d_i^*+d_j+d^*_j=2d$ for any $i,j$, we must have $k=1$. In other words $f$ projects nontrivially into ${\rm Hom}_W(V,H)_d$, which is enough to establish the theorem.
\end{proof}

In the proof we only really used that $d_1^*<d$, as we can use the result of \cite{MR1990017} mentioned in \ref{regularcriterion} to get regularity of $d$ from this (but regularity of $d$ alone is not sufficient for the proof). So in fact $W$ is a duality group precisely when $d_1^*<d$. This inequality is clear for real reflection groups since $d^*_1=d-2$ (one can also establish directly that $d$ is a regular number in the real case using the theory of Coxeter elements).

Once we have one regular vector in $Z(I_d)$, any nonzero vector in here will be regular. The pairing defined in terms of ${\rm ev}_v$ depends only on the orbit of $v$, and choosing an orbit corresponds to choosing a generator of ${\rm Hom}_W(V,H)_d$. Put geometrically, $\{f_1=...=f_{n-1}=0\}$ is the preimage in $V$ of the $f_n$-axis in $V\!/\!\!/W$; by regularity it consists of $|G|/d$ lines through the origin. The coordinate ring of this set is $S/I_d$ (the reducedness of $S/I_d$ is equivalent to the regularity of $d$, since $\spec (S/I_d)$ always has degree $|G|/d$ in $V$).


\section{K\"ahler Differentials and Derivations}\label{DerSection}

The reader can consult \cite[Chapter 16]{MR1322960} for (much) more detail on differentials and derivations. In short, the module of K\"ahler differentials (or $1$-forms) $\Omega^1_{R/\CC}$ of a $\CC$ algebra $R$ is the recipient of the universal $\CC$-linear derivation $d:R\to \Omega^1_{R/\CC}$. This universal property implies that the dual module ${\rm Hom}_R(\Omega^1_{R/\CC},R)$ is canonically isomorphic to the module of derivations $\Der(R,R)$. 

In our context, the inclusion $S^W\to S$ gives rise to a Jacobian map $J:\Der(S,S)\to \Der(S^W,S)$ by restriction. Choosing bases and writing $J$ as a matrix recovers the usual Jacobian matrix.

We can also consider the (free) $S^W$-module $\Der(S,S)^W$ of derivations which are invariant for the natural $W$ action. 
Extending scalars we get a natural map $M:\Der(S,S)^W\otimes_{S^{W}}S\to \Der(S,S)$. After choosing basis here, $M$ corresponds to the so-called ``coefficient matrix'' of a set of basic derivations, see \cite[definitions 4.11 and 6.67]{MR1217488}.

It is always the case that $(\Omega^1_{S/\CC})^W\cong \Omega^1_{S^W/\CC}$. Our first application of the duality pairing is that the corresponding fact for derivations characterises duality groups.

\begin{theorem}\label{derthm}
There is a natural map
\[
D\otimes S^W: \Der(S,S)^W\longrightarrow\Der(S^W,S^W)(d)
\]
which is an isomorphism if and only if $W$ is a duality group.
\end{theorem}

In other words: for duality groups there is a canonical degree $d$ isomorphism between the $W$-invariant derivations on $S$ and derivations on $S^W$. This means that the duality pairing can be rephrased as an $S^W$-bilinear pairing between $\Der(S,S)^W$ and $\Omega^1_{S^W/\CC}$, that is, between invariant vector fields on $V$ and $1$-forms on $V/\!\!/W$. 
This may turn out to be the correct perspective.

It's not difficult to see that under the duality condition $\Der(S^W,S^W)(d)$ and $\Der(S,S)^W$ are both free $S^W$ modules generated in the same degree, so it seems likely that this has been observed before. In any case, the canonical isomorphism constructed here will be very useful to us.

\begin{proof}
We just make the identifications
\[
\Der(S,S)^W \cong (V\otimes S)^W\cong (V\otimes H)^W\otimes S^W,
\]
as well as
\[
\Der(S^W,S^W) \cong ((V^*\otimes H)^W)^*\otimes S^W.
\]
After which the map in the theorem is defined by commutativity of the following diagram
\[
\begin{tikzcd}[column sep=20mm]
	\Der(S,S)^W\ar[r] \ar[d, equal, "\wr"']&  \Der(S^W,S^W) (d) \ar[d, equal, "\wr"]\\
	  (V\otimes H)^W\otimes S^W\ar[r,"D\otimes S^W"] & ((V^*\otimes H)^W)^*\otimes S^W(d) 
\end{tikzcd}
\]
where $D$ is the map defined just below theorem \ref{thm1}. According to that theorem, $D$ is an isomorphism precisely when $W$ is a duality group.
\end{proof}

We are ready to construct an important map which goes in the reverse direction to the Jacobian map. Assume that $W$ is a duality group. Extending scalars produces an isomorphism $D\otimes S: \Der(S,S)^W\otimes_{S^W}S \cong\Der(S^W,S)(d)$, and we define
\[
K=M (D\otimes S)^{-1} :\Der(S^W,S)(d) \longrightarrow \Der(S,S)  
\]
Note that everything above is $W$-equivariant. In particular, $JK$ is an equivariant endomorphism of $\Der(S^W,S)$, and so when written as a matrix the entries of $JK$ are invariant. Thus ${\rm det}(JK)$ is an invariant polynomial which is divisible by ${\rm det}(J)$, which therefore has to be divisible by $\Delta$. Since it has the same degree as $\Delta$, they coincide up to a scalar. Moreover, it follows that ${\rm det}(K)$ is the reduced equation for the critical locus (with $M$ the same argument applies to any reflection group, and this a particular case of Gutkin's theorem \cite{MR0314956}).

\section{Matrix Factorisations}

This section contains an extremely brief introduction to the theory of matrix factorisations. A readable account of the basics   can be found in \cite[Section 3]{MR2343380}.

Let $R$ be a commutative ring and $f$ be an element of $R$.  
A matrix factorisation of $f$ consists of a pair of projective $R$-modules $F_0$ and $F_1$ with a pair of $R$-linear maps $\phi:F_0\to F_1$ and $\psi:F_1\to F_0$ for which the compositions $\psi\phi$ and $\phi\psi$ are both multiplication by $f$. When $R$ is graded and $f$ is homogeneous we assume further that $\phi$ and $\psi$ are homogeneous of degree zero and $|f|$ respectively.

Matrix factorisations were introduced by Eisenbud \cite{MR570778} to investigate the stable behaviour of modules over a hypersurface singularity. It is in this case, when $R$ is regular, that these objects are most useful: then the homotopy category of matrix factorisations carries important geometric meaning. However, we will construct matrix factorisations over the singular ring $S/I_d$ (over a general complete intersection, matrix factorisations control the stable behaviour of modules of complexity one).

Crucially, when $f$ is a non-zero-divisor the $2$-periodic sequence
\[
\cdots \to F_0\otimes_R(R/f) \xto{\phi\otimes 1}F_1\otimes_R(R/f) \xto{\psi\otimes 1}F_0\otimes_R(R/f) \to \cdots
\]
obtained by reducing modulo $f$ is an exact sequence of free $R/f$-modules. If $f$ is homogeneous of some degree the sequence is only quasi-periodic: shifting the complex by $2$ results in a grading shift by $|f|$. Let it be emphasised that matrix factorisations do not simply provide a criterion for exactness, they are extremely important objects in their own right.

\subsection{}\label{MFcriterion} 
Assume again that $f$ is a non-zero-divisor. Let $F_0$ and $F_1$ be free $R$-modules of equal finite rank, and assume that $\phi:F_0\to F_1$ and $\psi:F_1\to F_0$ are linear maps such that $\psi\phi=f\cdot 1_{F_1}$. As Eisenbud notes in  \cite{MR570778}, it is then automatically true that $\phi\psi=f\cdot 1_{F_0}$, so this data determines a matrix factorisation.

\section{Resolving $\Der(A,A)$ and $\Omega^1_{A/\CC}$} 
\label{resolutions}
The results of this section 
were the motivation for the construction of the duality group pairing. 
Associated to the algebra map $S^W\to S$ there is a well-known presentation for the  module of  
\[
\Omega^1_{S^W/\CC} \otimes_{S^W} A\to\Omega^1_{S/\CC}\otimes_S A\to \Omega^1_{A/\CC}\to0.
\]
How can this be continued to a free resolution?

We prefer to phrase things dually in terms of derivations. That is, we'll continue the exact sequence
\[
0\to \Der(A,A)\to \Der(S,A)\to \Der(S^W,A)
\]
periodically to the right, using a matrix factorisation. The desired resolution of the K\"ahler differentials can be recovered by applying ${\rm Hom}_A(-,A)$.

Essentially, $J$ and $K$ form the matrix factorisation we are after. However, differentiating introduces constants which makes this not quite true. To fix this we introduce the reduced Jacobian $\overline{J}: \Der(S,S)\to \Der(S^W,S)$ which, considered as a map $V\otimes S\to {\rm Hom}_W(V,H)^*\otimes S$, is given by $\overline{J}(v\otimes f)=\sum \sigma^*\otimes \sigma(v)f$, where the sum is over a basis $\{\sigma\}$ of ${\rm Hom}_W(V,H)^*$.

Let us first explain why, modulo the ideal $I_d$, $\overline{J}$ agrees with $J$ up to post-composition with an invertible matrix of scalars.

More precisely, if $\sigma \in {\rm Hom}_W(V,H)$ and $\tilde{\sigma} = \sum u^*\sigma(u)$ is the corresponding polynomial in $S^W$, then we will show that $ \frac{\partial  \tilde{\sigma}}{\partial v} = |\sigma| \sigma(v)$ modulo $I_d$, so this matrix of scalars is multiplication by degree on ${\rm Hom}_W(V,H)^*$. For obvious degree reasons, it suffices to do this modulo the full Hilbert ideal $I$. Since $S=H\oplus I$ as $W$ representations, we can write $\frac{\partial  \tilde{\sigma}}{\partial v} = h_\sigma(v)+i_\sigma(v)$ for equivariant maps $h_\sigma :V\to H$ and $i_\sigma :V\to I$. Since $\sum v^*\frac{\partial  \tilde{\sigma}}{\partial v} = |\sigma|\tilde{\sigma}$ the map $k_\sigma=|\sigma|h_\sigma -\sigma:V\to H$ satisfies $\sum v^* k_\sigma(v) = |\sigma |\sum v^*i_\sigma(v)$. But, if we apply the Reynolds operator to the right-hand-side here, we see that this is a decomposable invariant polynomial, even though the Chevalley-Sheppard-Todd theorem says the map ${\rm Hom}_W(V,H)\to S^W$ is injective and has only indecomposables in its image. Hence in fact $k_\sigma =0$.

In particular, $\overline{J}\otimes_S A$ has the same kernel as $J\otimes_S A$, and any matrix factorisation involving $\overline{J}\otimes_S S/I_d$ is isomorphic to one involving $J\otimes_S S/I_d$.

\begin{theorem}\label{MF} Let $W$ be a duality group equipped with a choice of highest degree basic invariant $f_n$. The maps $\overline{J}$ and $K$ defined above
\[
    \begin{tikzcd}
                \Der(S,S)  \ar[r, shift left=3pt] &   \Der(S^W,S) \ar[l, shift left=3pt]
    \end{tikzcd}
\]
result in a matrix factorisation of $f_n$ after reducing modulo $f_1,...,f_{n-1}$. Hence applying  $-\otimes_S A$ we obtain a $2$-periodic minimal free resolution of $\Der(A,A)$. Dualizing, we also obtain a $2$-periodic minimal free resolution of $\Omega^1_{A/\CC}$.
\end{theorem}

The appearance of these matrix factorisations is quite surprising. There is no obvious reason that the $A$-module $\Omega^1_{A/\CC}$ should have complexity one (i.e. a resolution by free modules of bounded rank),
this is far from the generic situation.

As remarked above, up to a scalar ${\rm det}(JK)$ is $\Delta$, so we recover the fact that when $W$ is a duality group, $\Delta$ is monic in $f_n$ of degree $n$.

\begin{proof}
A direct calculation verifies that $\overline{J}K=f_n\cdot 1_{\Der(S^W,S/I_d)}$, it goes as follows. Fix an element $\partial$ of ${\rm Hom}_W(V,H)^*$, considered as a derivation in $\Der(S^W,S)$. The corresponding element $\phi$ of ${\rm Hom}_W(V^*,H)$ by definition satisfies $\sum_v \phi(v^*)\sigma(v)=\partial(\sigma)f_n$ after reducing to $S/I_d$, where the sum is taken over some basis $\{v\}$ of $V$, and $\sigma$ is any element of ${\rm Hom}_W(V,H)$.  The map defined above takes $\phi$ to $K(\partial)= \sum v\otimes  \phi(v^*)$ now considered as an element of $V\otimes S\cong \Der(S,S)$. The reduced Jacobian then takes $ \sum v\otimes  \phi(v^*)$ to
\[
\overline{J}K(\partial) =\sum_\sigma \sum_v  \sigma^* \otimes  \sigma(v)\phi(v^*) = \sum_\sigma \sum_v  \sigma^* \otimes  \partial(\sigma)f_n,
\]
summing over some basis $\{\sigma\}$ of ${\rm Hom}_W(V,H)$. 
From here we see $\overline{J}K(\partial)(\sigma)=\partial(\sigma)f_n$ holds in $S/I_d$. After this, the observation of paragraph \ref{MFcriterion} completes the proof that $\overline{J}$ and $K$ determine a matrix factorisation of $f_n$.

The remaining assertions follow from the presentation of $\Der(A,A)$ as the kernel of $J:\Der(S,A)\to \Der(S^W,A)$, and of $\Omega^1_{A/\CC}$ as the cokernel of $J^*:\Omega^1_{S^W/\CC}\otimes_{S^W}A\to \Omega^1_{S/\CC}\otimes_S A$.
\end{proof}

\begin{remark}\label{BessisRemark}
As pointed out by the referee, this bears a close resemblance to Bessis' decomposition of the discriminant matrix in \cite[theorem 2.4 (iv)]{MR3296817}. The discriminant matrix of \cite[definition 6.67]{MR1217488} is the matrix $M_\Delta$ over $S^W$ obtained from $JM: \Der(S,S)^W \to \Der(S^W,S^W)$ by choosing bases. Bessis shows that $M_\Delta = M_0+ f_nM_1$, where $M_0$ has entries in $\CC[f_0,...,f_{n-1}]$ and $M_1$ is lower-triangular with non-zero scalars along the diagonal.

Assuming theorem \ref{derthm}, Bessis' decomposition follows from the the existence of the matrix factorisation above. Conversely, one can almost use theorem \ref{derthm} to reconstruct the matrix factorisation from Bessis' decomposition, except for the possibility of scalars in $M_0$.

Since Bessis' decomposition has been used in \cite[lemma 3.8 and remark 3.2]{2015arXiv151101608K} to construct flat coordinates on the orbit space $V{\!/\!\!/} W$, it would be very interesting to investigate the connection between the construction in this paper and those of loc.~cit.
\end{remark}


\begin{example*} Already interesting is the case of the symmetric group $\mathfrak{S}_{n+1}$ acting on  $V=\left\{(\alpha_0,...,\alpha_n) : {\scriptstyle\sum} \alpha_i=0\right\}$. Then $S=\CC[x_0,...,x_n]/({\scriptstyle \sum} x_i=0)$ where $x_i$ is dual to $(0,...,1,...,0)$. It is classical that $S^{\mathfrak{S}_{n+1}}$ is the polynomial ring $\CC[\sigma_2,...,\sigma_{n+1}]$ on the elementary symmetric polynomials $\sigma_k=(-1)^{k-1}\sum_{i_1<...<i_k}x_{i_1}\cdots x_{i_k}$. So the algebra of coinvariants is $A=S/(\sigma_2,...,\sigma_{n+1})$, and the highest degree is $d=n+1$. The duality pairing will naturally match $\sigma_{i+2}$ with $\sigma_{n+1-i}$. We need to choose bases to get matrices. If $v=\frac{1}{n+1}(1,...,1)$ and $v_i=(0,...,1,...,0)-v$ then when we differentiate we get $\frac{\partial \sigma_k}{\partial{v_i}}= x_i^{k-1}$ modulo $\sigma_2,...,\sigma_{n}$. So we take as bases $v_1,...,v_n$ for $V$ and $\sigma_2,...,\sigma_{n+1}$ for ${\rm Hom}_{\mathfrak{S}_{n+1}}(V,H)$, and we get
\[J=
\begin{pmatrix}
x_1 &\cdots & x_n\\
\vdots & & \vdots\\
x_1^n &  \cdots& x_n^n \\
\end{pmatrix}
\quad \text{ and }\quad K=
\begin{pmatrix}
x_1^n -x_0^n &\cdots & x_1- x_0\\
\vdots & &  \vdots\\
x_{n}^n -x_0^n &  \cdots& x_n-x_0 \\
\end{pmatrix}
\]
as matrices with entries in $S/I_d$ (calculating $K$ is made easier by the observation made in Corollary \ref{coxeterresolution}). Even knowing these matrices, it is combinatorially quite involved to check directly that they form a matrix factorisation of $\sigma_{n+1}$. 
\end{example*}

Going back to an arbitrary duality group, we can read from the matrix factorisation of theorem \ref{MF} the precise dimensions of all the graded components of $\Omega^1_{A/\CC}$ and $\Der(A,A)$.

First we need a few facts about duality from commutative algebra, which are properly explained in \cite[Chapter 21]{MR1322960}, for example.

If $M$ is graded vector space then its Hilbert series is by definition $H_M(t)=\sum_i (\dim M_i)t^i$. We use the facts that $H_{M(n)}(t)=t^nH_M(t)$, and $H_{M^*}(t)=H_M(t^{-1})$, and $H_{-}(t)$ is ``additive on exact sequences'' in the obvious sense.

Since $A$ is a finite-dimensional complete intersection ring it is a Frobenius algebra: there is a canonical bimodule isomorphism $A^*(N)\cong A$, where $N$ is the socle degree of $A$. In the context of reflection groups $N=(d_1-1)+\cdots+(d_n-1)$ is the number of nontrivial reflections in $W$.  It follows that for every module $M$ there is a natural isomorphism $M^*(N)\cong {\rm Hom}_A(M,A)$. In particular $(\Omega^1_{A/\CC})^*(N)\cong \Der(A,A)$, and so $H_{\Der(A,A)}(t)= t^{N} H_{\Omega^1_{A/\CC}}(t^{-1})$.

\begin{theorem}\label{HilbertSeries}
The Hilbert series of the graded modules  $\Omega^1_{A/\CC}$ and $\Der(A,A)$ are given by
\[
H_{\Omega^1_{A/\CC}}(t)\ =\ 
\left(\sum_{i\leq n}\frac{t-t^{d_i}}{1-t\ \ }\right) \cdot \left(\prod_{i<n}\frac{1-t^{d_i}}{1-t\ \ }\right),
\]
\[
H_{\Der(A,A)}(t)\ =\ 
\left(\sum_{i\leq n}\frac{t^{d_i^*}-t^{d-1}}{1-t}\right) \cdot \left(\prod_{i<n}\frac{1-t^{d_i}}{1-t\ \ }\right).
\]
In particular the total dimension is
\[
{\rm dim}_{\CC}\, \Omega^1_{A/\CC} = \frac{N|W|}{d} = {\rm dim}_{\CC} \,\Der(A,A).
\]
\end{theorem}
\begin{proof}
We'll first do this for the K\"{a}hler differentials. From the exact sequence
\[
0\to{\rm ker}(J^*)\to \Omega^1_{S^W/\CC}\otimes_S A\xrightarrow{\ J^*\ }
\Omega^1_{S/\CC}\otimes_S A\to {\rm coker}(J^*)\to 0,
\]
there is an equality of Hilbert series
\[
H_{\Omega^1_{S/\CC}\otimes_S A}(t)- H_{\Omega^1_{S^W/\CC}\otimes_S A}(t) = H_{{\rm coker}(J^*)} (t)- H_{{\rm ker}(J^*)}(t).
\]
But ${\rm coker}(J^*)= \Omega^1_{A/\CC}$ and the matrix factorisation gives us an isomorphism ${\rm ker}(J^*)\cong {\rm coker}(J^*)(d)$ induced by $K$. Also note that $H_{\Omega^1_{S/\CC}\otimes_S A}(t)= nt H_A(t)$ and $H_{\Omega^1_{S^W/\CC}\otimes_S A}(t)= ({\scriptstyle \sum_i} t^{d_i}) H_A(t)$. From this we deduce
\[
ntH_A(t)- ({\scriptstyle \sum_i} t^{d_i}) H_A(t)= H_{\Omega^1_{A/\CC}}(t)  - t^dH_{\Omega^1_{A/\CC}}(t).
\]
Keeping in mind that $H_A(t)=\prod_i\frac{1-t^{d_i}}{1-t\ }$, this can be rearranged to give the expression of the theorem.

We can deduce the expression for $H_{\Der(A,A)}(t)$ using the equation  $H_{\Der(A,A)}(t)= t^{N} H_{\Omega^1_{A\CC}}(t^{-1})$.

Finally, setting $t=1$ in either polynomial gives $[(d_1-1)+\cdots+(d_n-1)]\cdot d_1\cdots d_{n-1}=\frac{N|W|}{d}$.
\end{proof}

When $W$ is a real reflection group the matrix factorisation takes a particularly striking form. In this case we may equip $V$ with a perfect $W$-invariant bilinear (as opposed to linear-antilinear) form  $\,V\otimes V\to \CC(-2)$. There is canonical one when $W$ is the Weyl group of a flag manifold, for instance. We obtain a $\CC W$-isomorphism $V(2)\cong V^*$, and hence an isomorphism
\[
\Omega^1_{S/\CC} \cong V^*\otimes S \cong V\otimes S\, (2) \cong {\rm Der}_{\CC}(S,S)(2).
\]
More interestingly, choosing a highest degree basic invariant $f_n$ gives us an isomorphism
\[
\begin{tikzcd}[column sep=-1mm, row sep=1mm]
\Omega^1_{S^W/\CC} \ \cong & {\rm Hom}_W(V, H)\otimes S^W \ \cong  & \!\!\!\!{\rm Hom}_W(V^*, H)^*\otimes S^W (d) \ \cong & \\
& & \ \ {\rm Hom}_W(V, H)^*\otimes S^W (d+2) \ \cong  & {\rm Der}_{\CC}(S^W,S^W)(d+2).
\end{tikzcd}
\]
\begin{remark}\label{ShephardRemark}
This very similar to \cite[theorem 6.121(2)]{MR1217488}, where an isomorphism $\Der(S,S)^W\cong (\Omega_{S/\CC}^1)^W(d_1)$ is constructed for any Shephard group $W$ (all Coxeter groups are Shephard groups, and all Shephard groups are duality groups). If we combine this with theorem \ref{derthm} then we recover the Coxeter case of that result. This suggests that, most likely, what comes below can be extended to the case of Shephard groups in some form.
\end{remark}

Dual to the reduced Jacobian is the canonical map $\overline{J}^*:\Omega^1_{S^W/\CC}\otimes_{ S^W}S\to \Omega^1_{S/\CC}$, which up to an invertible matrix agrees with the usual $J^*$. Under the identifications just described this is actually the map $K:\Der(S^W,S)\to \Der(S,S)$ of section \ref{DerSection}. Hence

\begin{corollary}\label{coxeterresolution} Let $W$ be a real reflection group equipped with a perfect $W$-invariant bilinear form and a choice of highest degree basic invariant $f_n$. Making the above identifications, the reduced Jacobian and its dual:
\[
\begin{tikzcd}[row sep=3mm]
\Omega^1_{S^W/\CC} \otimes_{S^W} S \ar[r] \ar[d, equal, "\wr"'] & \Omega^1_{S/\CC}   \ar[d, equal, "\wr"]   \\
                {\rm Der}_{\CC}(S^W,S) &    {\rm Der}_{\CC}(S,S) \ar[l]
\end{tikzcd}
\]
determine a matrix factorisation of $f_n$ when reduced to $S/I_d$. In particular, the 2-periodic resolution of $\Omega^1_{A/\CC} $ and $ \Der(A,A)$ thus obtained is isomorphic to its own dual, shifted by one.
\end{corollary}

We obtain the following curious fact. Recall that the first syzygy ${\rm Syz}^1_A M$ of a graded $A$ module $M$ is by definition the kernel of a surjection $A^k\to M$ from a free module with the minimal possible number of generators.

\begin{corollary} If $A$ is the algebra of coinvariants associated to a Coxeter group then
\[
{\rm Syz}_A^1\,\Omega^1_{A/\CC} =\Der(A,A)(d+2) \quad \text{ and } \quad {\rm Syz}_A^1\,\Der(A,A) =\Omega^1_{A/\CC}(-2).
\]
Moreover we have
\[
{\rm dim}_{\CC}\, \Omega^1_{A/\CC} = \frac{n|W|}{2} = {\rm dim}_{\CC} \,\Der(A,A).
\]
\end{corollary}



Finally, returning to an arbitrary duality group, let us point out that being a complete intersection, the Andr\'{e}-Quillen cohomology of $A$ can be computed as the cohomology of the complex $\Der(S,A)\xrightarrow{J} \Der(S^W,A)$, which lives in cohomological degree $0$ and $1$ (see \cite[Section 10]{MR0257068} or \cite[Section 8]{MR2355775} for this standard computation). From the matrix factorisation we get a canonical isomorphism $H^0_{AQ}(A,A)\cong H^1_{AQ}(A,A)(d)$ induced by $K$, so from theorem \ref{HilbertSeries} we can read off the entire bigraded Hilbert series of the Andr\'{e}-Quillen cohomology (we use this observation below, but with different grading conventions). 

In this context, the first  Andr\'{e}-Quillen cohomology group is also known as the Tjurina module $T^1_A={\rm coker}(J)$. The dimension ${\rm dim}_\CC T^1_A$ is the Tjurina number of the singularity defined by $A$. This is the number of parameters in the base of the miniversal deformation of $A$. This is explained in \cite[chapter 6]{Looijenga}. Theorem \ref{MF} produces a canonical isomorphism $\Der(A,A)\cong T^1_A(d)$, hence the Tjurina number of $A$ is $\frac{N|W|}{d}$ (which is $\frac{n|W|}{2}$ in the Coxeter case). Moreover, this allows one to produce a basis for $T^1_A$ (this calculation is entirely feasible in the case of the symmetric group, for instance), and from this data one can explicitly construct the miniversal deformation of $A$, see loc.~cit. It would be interesting to try to identify certain natural deformations of $A$ within this space.

\section{Flag Manifolds}
\label{flagvarieties}

A number of authors have investigated the rational homotopy types of flag manifolds, especially what can be said about classifying their self-maps, e.g.  \cite{MR1471062,MR894562, MR710764}. The purpose of this section is to briefly point out that in the Weyl case, the calculations above can be interpreted in these terms.

Thus, let $X=G/B$ be a complete flag manifold for some semi-simple complex algebraic group $G$ and Borel subgroup $B$. Let $T$ be a maximal torus in $B$ of rank $n$ and set $V={\rm Lie}(T)$. In this situation we should place $V$ in degree $2$, doubling the gradings from the previous sections. The associated Weyl group $W$ acts on $V$ as a duality group, so the above results apply. The necessary rational homotopy theory background can be found in \cite{MR1802847}.

Each of the homotopy groups $\pi_n(X)$ is a finitely generated abelian group, so we may consider the graded vector space of complexified rational homotopy groups $\pi_{*}(X)_\CC=\pi_{*}(X)\otimes_{\mathbb{Z}}\CC$ (there's no need for us to consider the Lie algebra structure here).

According to the famous \emph{Borel picture} there is a canonical isomorphism
\[
A=\CC[V]/\CC[V]_+^{W}\xrightarrow{\ \cong\ } {\rm H}^{*}(X;\CC).
\]
This is largely the reason for our interest in the coinvariant algebra.  In this context the Frobenius condition used in section \ref{resolutions} is equivalent to Poincar\'e duality.

A quick consequence of the previous section is that the Halperin conjecture holds for flag manifolds: the Serre spectral sequence associated to any fibration with fibre $X$ must collapse at the second page. This is known to be equivalent to the fact that $\Der(A,A)_{<0}=0$, which can be read from theorem \ref{HilbertSeries}. This was proven by Shiga and Tezuka in \cite{MR894562} (but actually has been known for any simply connected K\"{a}hler manifold since \cite{MR0087184}).

Because $X$ is a formal space in the sense of \cite[p.315]{MR0646078}, 
the graded algebra $A$ completely determines the rational homotopy type of $X$, and this has been well exploited in the literature. For instance, there is an isomorphism $\pi_{i}(X)_\CC\cong {\rm H}^{-i}_{\rm AQ}(A,\CC)$ as in \cite[remark 3.9]{MR2132759}, although the connection between Andr\'e-Quillen cohomology and rational homotopy theory is much older \cite{RationalHT}.  By convention $A$ is generated in degree two and the cohomological degree is negative. Under these conventions it is a standard computation that $ {\rm H}^{*}_{\rm AQ}(A,\CC)\cong  V\oplus {\rm Hom}_W (V,A)(-1)$, see \cite[Section 10]{MR0257068} or \cite[Section 8]{MR2355775}. This calculation of the rational homotopy groups of a flag manifold is presumably well-known.


Now we use the previous sections to classify self-rational-equivalences of $X$. Denote by ${\rm aut}(X,1)$ the component of the identity in the mapping space ${\rm map}(X,X)$. Rational homotopy theory works best for simply connected spaces, so for simplicity we take the universal cover $\widetilde{\rm aut}(X,1)$, or equivalently, we only consider the homotopy groups $\pi_{i}({\rm aut}(X,1))_\CC$ for $i\geq 2$.

For each $i\geq 2$ there is an isomorphism 
$
\pi_{i}({\rm aut}(X,1))_\CC \cong {\rm H}^{-i}_{\rm AQ}(A,A),
$ 
obtained by complexifying \cite[theorem 3.8]{MR2132759} (and using formality $A\simeq A^*(X)\otimes_{\mathbb{Q}}\CC$). Under these grading conventions ${\rm H}^{*}_{\rm AQ}(A,A)$ is the homology of  the complex $\Der(S,A)\xrightarrow{J} \Der(S^W,A)(-1)$. Using our matrix factorisation as we did at the end of section \ref{resolutions} we get ${\rm H}^{*}_{\rm AQ}(A,A)= \Der(A,A)\oplus \Der(A,A)(-2d-1)$. In particular, by theorem \ref{HilbertSeries}
\[
H_{{\rm H}^{*}_{\rm AQ}(A,A)}(t)\ =\ 
(1+t^{-2d-1})\left(\sum_{i\leq n}\frac{t^{2d_i^*}-t^{2d-2}}{1-t^2}\right) \cdot \left(\prod_{i<n}\frac{1-t^{2d_i}}{1-t^2\ \ }\right).
\]
The component coming from the $1$ at the beginning only produces positive powers of $t$, so we can ignore it below. Then multiplying out gives


\begin{theorem}
For $i\geq 2$, the dimension of $\pi_{i}({\rm aut}(X,1))_\CC $ is the coefficient of $t^{-i}$ in the Laurent polynomial
\[ 
\left(\sum_{i\leq n}\frac{t^{-1-2d_i}-t^{-3}}{1-t^{2} }\right) \cdot \left(\prod_{i<n}\frac{1-t^{2d_i}}{1-t^2\ \ }\right).
\]
\end{theorem}

Any topological monoid is rationally homotopy equivalent to a product of Eilenberg-Maclane spaces, hence the above theorem completely determines the rational homotopy type of $\widetilde{\rm aut}(X,1)$; it is a product of odd dimensional (rational) spheres, the number in each dimension being given by the coefficients of the above Laurent polynomial. In \cite{MR1471062} Smith also calculates the rational homotopy type of $\widetilde{\rm aut}(X,1)$, at least for the algebraic groups $GL(n)$ and $Sp(n)$. His expression for the number of spheres is less explicit, and it's not clear how it relates to ours.



\bibliographystyle{amsplain}
\bibliography{Dualitybib}
\end{document}